\documentclass[11pt, reqno]{amsart}
\usepackage[margin=3.2cm]{geometry}                
\usepackage{bigints, cancel, enumerate}
\usepackage[all, cmtip]{xy}
\usepackage{graphicx}
\usepackage{amssymb}
\usepackage{epstopdf}
\usepackage{xcolor}
\usepackage{ulem}
\usepackage{subfig}
\usepackage{cleveref}
\usepackage{abstract}
\usepackage{enumerate, subfig}
\usepackage[all]{xy}

\usepackage{amsmath, amsthm, amssymb, stmaryrd}
\theoremstyle{plain}
\newtheorem{proposition}{Proposition}[section]
\newtheorem{lemma}[proposition]{Lemma}
\newtheorem{theorem}[proposition]{Theorem}
\numberwithin{equation}{section}
\DeclareMathOperator{\Arcsin}{Arcsin}
\newcommand{\CP}{\mathbb{CP}}
\newcommand{\real}{\mathbb{R}}
\newcommand{\comp}{\mathbb{C}}

\newcommand{\intg}{\mathbb{Z}}

\newcommand{\mcal}[1]{{\mathcal{#1}}}
\newcommand{\half}{\frac{1}{2}}
\newcommand{\delt}{\partial_t}
\newcommand{\knp}{\owedge}
\newcommand{\Szero}{\CP^1\times\CP^1}
\newcommand{\Sone}{\CP^2\#\overline{\CP}^2}
\newcommand{\opint}{(-T, T)}

\newcommand{\Rm}{R_m}
\DeclareMathOperator{\Ric}{Ric}
\DeclareMathOperator{\tsRic}{\Ric^{\circ}}
\DeclareMathOperator{\Hess}{Hess}

\DeclareMathOperator{\Vol}{Vol}
\DeclareMathOperator{\dVol}{dV}
\DeclareMathOperator{\Met}{Met}

\DeclareMathOperator{\trace}{trace}
\DeclareMathOperator{\id}{id}
\DeclareMathOperator{\SU}{SU}
\DeclareMathOperator{\U}{U}
\DeclareMathOperator{\GL}{GL}
\DeclareMathOperator{\pr}{pr}
\DeclareMathOperator{\Ker}{Ker}
\newcommand{\cd}{\nabla}
\newcommand{\lact}{\curvearrowright}\newcommand{\ract}{\curvearrowleft}

\title[]{Constant scalar curvature metrics \\on Hirzebruch surfaces}

\email{otoba@math.keio.ac.jp}
\date{\today}
\author[]{Nobuhiko Otoba}
\address{Keio University, 3-14-1 Hiyoshi, Kohoku-ku, Yokohama, Kanagawa 223-8522, Japan}
\begin{document}

\maketitle

\begin{abstract}
We construct smooth Riemannian metrics with constant scalar curvature on each Hirzebruch surface.  
These metrics respect 
the complex structures, fiber bundle structures, and Lie group actions of cohomogeneity one 
on these manifolds.  
Our construction is reduced to an ordinary differential equation called Duffing equation.  
An ODE for Bach-flat metrics on Hirzebruch surfaces with large isometry group is also derived.  
\end{abstract}

\section{Introduction and main results}\label{Results}
For each integer $m\ge 0$, Hirzebruch introduces a simply-connected complex surface $\Sigma_m$, 
now called \textit{the $m$-th Hirzebruch surface} \cite{Hir}.  
The first two surfaces $\Sigma_0$ and $\Sigma_1$ 
are known to be biholomorphically equivalent to $\Szero$ and $\Sone$, respectively, 
the latter being the connected sum of two complex projective planes with usual and inverse orientations.  
On one hand, 
each $\Sigma_m$ of these surfaces has a structure of $\CP^1$ bundle over $\CP^1$.  
On the other hand, when $m\ge 1$, 
$\Sigma_m$ admits an effective action of the Lie group $\U(2)/\left( \intg/m\intg\right)$, with orbit space a compact interval of real numbers.  
Hirzebruch surfaces are thus both locally trivial fiber bundles and cohomogeneity-one manifolds for $m\ge 1$.  

After Page \cite{Pag} constructed an Einstein metric on $\Sone$, 
B\'erard-Bergery \cite{BB} not only translated the construction into mathematics, 
but also characterized and generalized Page metric from the perspective of cohomogeneity-one Riemannian geometry.  
Following his work, 
several geometric structures on Hirzebruch surfaces with high symmetry, 
such as Einstein-Weyl structures \cite{MPPS} and extremal K\"ahler metrics \cite{HS} to mention a few, 
were constructed.

In this paper, we look for critical metrics on each Hirzebruch surface $\Sigma_m$ under the following assumptions ($m\ge 1$): 
\begin{enumerate}[(I)]
\item The fiber bundle projection $\pi_m: \Sigma_m\to\CP^1$ is a Riemannian submersion onto $\CP^1$ equipped with a metric of area $\pi$.  
\item The action $\U(2)/\left( \intg/m\intg\right)\lact\Sigma_m$ of cohomogeneity one is by isometries.  
\end{enumerate}
Normalization of area in condition (I) is to exclude homothety.  
For consequences of these assumptions, see Proposition \ref{consequences}.  
Critical metrics here are meant to be smooth Riemannian metrics satisfying Euler-Lagrange equations of curvature functionals.  
We focus our attention to the following functionals, which might be summarized as linear and quadratic curvature functionals (cf. \cite[Chapter 4]{Bes3}).  
For their definitions, 
let $M$ be a $4$-dimensional closed manifold and $\Met(M)$ the space of all $C^{\infty}$ metrics on $M$.  
The normalized Einstein-Hilbert functional $E: \Met(M)\to\real$ is defined by 
\[
E(g)=\frac{\int_{M} R_g\dVol_g}{\sqrt{\int_M\dVol_g}}, 
\]
where $R_g$ and $\dVol_g$ stand for the scalar curvature and volume element of $g$; 
a metric  $g\in\Met(M)$ is critical with respect to $E$ if and only if $g$ is Einstein.  
We also consider Yamabe functional $Y$, the restriction of $E$ to a conformal class on $M$; 
its critical points are precisely metrics of constant scalar curvature (cf. \cite{LP}).  
By a quadratic curvature functional, we mean a linear combination $a\mcal{W}+b\rho+c\mcal{S}: \Met(M)\to\real$ 
with constant coefficients $a, b, c\in\real$ of the following three functionals $\mcal{W}$, $\rho$, and $\mcal{S}$ 
defined respectively by
\begin{align*}
\mcal{W}(g)=\int_M\vert W_g\rvert^2\dVol_g, 
&&\rho(g)=\int_M\vert \Ric_g\rvert^2\dVol_g, 
&&\text{and}
&&\mcal{S}(g)=\int_MR_g^2\dVol_g.  
\end{align*}
Here, $W_g$ and $\Ric_g$ stand for the Weyl tensor and Ricci tensor of $g$, 
and we emphasize that $\lvert W_g\rvert$ and $\lvert\Ric_g\rvert$ are their tensor norms with respect to $g$.  
As Gursky and Viaclovsky point out in view of Chern-Gauss-Bonnet formula \cite{GV}, 
insofar as we are concerned with critical points of linear and quadratic curvature functionals in $4$-dimensions, 
it suffices to consider in addition to $E$ and $Y$ the following particular linear combination $\mcal{B}_t: \Met(M)\to\real$ defined by 
\[
\mcal{B}_t(g)=\mcal{W}(g)+t\mcal{S}(g)=\int_M\lvert W_g\rvert^2\dVol_g+t\int_MR_g^2\dVol_g
\]
for each real number $t$.  
Critical metrics of $\mcal{B}_t$-functional are said to be $B^t$-flat, 
and they have a tensorial characterization in terms of Ricci tensor and Bach tensor.  
In short terms, our objects to study are constant scalar curvature metrics, $B^t$-flat metrics, 
and Einstein metrics on Hirzebruch surfaces satisfying conditions (I) and (II).  

Our main results are the following.  
\begin{theorem}
\label{cscExistence}
For each integer $m\ge 1$ and each real number $R$, 
there exists a constant scalar curvature metric $g_m(R)$ on the $m$-th Hirzebruch surface $\Sigma_m$ 
satisfying conditions \textup{(I)} and \textup{(II)}.  
Scalar curvature of $g_m(R)$ is equal to the constant $R$ we have given.  
\end{theorem}

\begin{theorem}
\label{nonExistence}
Let $g$ be a Riemannian metric on $\Sigma_m$ satisfying conditions \textup{(I)} and \textup{(II)}, 
and assume $g$ is critical with respect to a linear or quadratic curvature functional.  
Then, either $g$ is Bach flat, or $g$ coincides with a metric of Theorem \ref{cscExistence}.  
These two cases are mutually exclusive.  
\end{theorem}

The constant scalar curvature metrics of Theorem \ref{cscExistence} are defined through an ordinary differential equation called Duffing equation, 
whose solutions are able to be analyzed in detail, 
whereas an ODE with movable essential singularities describes Bach-flat metrics on Hirzebruch surfaces 
with the $\U(2)/\left(\intg/m\intg\right)$-symmetry.  
Since a conformal deformation of Page metric \cite{Pag} satisfies conditions (I) and (II), 
the Bach-flat case of Theorem \ref{nonExistence} is nonempty.  
Other Bach-flat metrics are yet to be fully understood in this work.  

This paper is organized as follows.  
In Section \ref{Preliminaries}, we set up our situation to see that
all the necessary computations of curvature quantities 
are reduced to tensor calculations on the product manifold $S^3\times (-T, T)$.  
Taking it into account, in Section \ref{Tensors}, 
we calculate the scalar curvature and Bach tensor of certain Riemannian metrics on $S^3\times (-T, T)$.  
In Section \ref{CscMetrics}, 
proofs of Theorems \ref{cscExistence} and \ref{nonExistence} (\S \ref{MainProofs}) are 
followed by more detailed properties of the constant scalar curvature metrics of Theorem \ref{cscExistence} (\S \ref{FurtherProperties}).  
The Bach-flat equation mentioned previously is treated in Section \ref{BachFlat}.  


\section{Preliminaries}\label{Preliminaries}
The finite cyclic group $\intg/m\intg$ is written as $\Gamma_m$ in the sequel.  
In Sections \ref{bundle}, \ref{action}, and \ref{cKahler}, we see that 
the fiber bundle projections $\pi_m:\Sigma_m\to\CP^1$, 
the group actions $\U(2)/\Gamma_m\lact\Sigma_m$ 
and the conformal K\"ahler triplets $\left(J_m(f), g_m(f), \omega_m(f)\right)$ 
to be defined through a certain procedure have intimate relationships with each other.  
Thereafter, looking at particular open dense submanifolds of $\Sigma_m$ (Section \ref{odsub}), 
we show that every Riemannian metric satisfying conditions (I) and (II) has to be defined through this procedure (Section \ref{characterization}).  
Of our special attention is the first Hirzebruch surface $\Sigma_1$.  
Corresponding arguments for higher Hirzebruch surfaces $\Sigma_m$ require few modifications ($m\ge 2$).  

\subsection{Fiber bundle projections $\pi_m:\Sigma_m\to\CP^1$}\label{bundle}

Let $S^3=\{\left(z, w\right)\in\comp^2\mid\lvert z\rvert^2+\lvert w\rvert^2=1\}$ be the $3$-sphere.  
On $S^3$, the circle group 
$S^1=\{e^{i\theta}\in\comp\mid\theta\in\real\}$ 
acts freely on the right 
by componentwise multiplication $(z, w).e^{i\theta}=(ze^{i\theta}, we^{i\theta})$.  
The quotient map $p_1: S^3\to\CP^1$ onto its orbit space is a principal $S^1$ bundle called the Hopf fibration.  

We consider the $2$-sphere $S^2$ as the disjoint union of the cylinder $S^1\times(-T, T)$ and two points, 
where we introduce the usual differential structure by means of polar coordinates.  
The left $S^1$ action on $S^1\times(-T, T)$ defined by $e^{i\theta}.(e^{i\phi}, t)=(e^{i(\theta+\phi)}, t)$ 
extends to an effective smooth action $S^1\lact S^2$ fixing exactly these two points.  

Let $S^1$ act freely on $S^3\times S^2$ by $(\tilde{x}, \hat{x}).e^{i\theta}=(\tilde{x}.e^{i\theta}, e^{-i\theta}.\hat{x})$, 
with the corresponding principal $S^1$ bundle $q_1: S^3\times S^2\to\Sigma_1$.  
Since the projection $\pr_1:S^3\times S^2\to S^3$ onto the first factor is equivariant with respect to the $S^1$ actions, 
we have the following commutative diagram
\begin{equation}\label{principalmap}
\begin{split}
\xymatrix{
S^3\times S^2\ar[r]^-{\pr_1}\ar[d]_-{q_1}&S^3\ar[d]^-{p_1}\\
\Sigma_1\ar[r]^-{\pi_1}&\CP^1.  
}
\end{split}
\end{equation}
The associated bundle $\pi_1: \Sigma_1\to\CP^1$ is the first Hirzebruch surface as locally trivial smooth $S^2$ bundle over $\CP^1$.  

When $m\ge 2$, we denote by $\Gamma_m$ the subgroup 
$\{e^{i\frac{2\pi l}{m}}\mid l\in\intg\}$ 
of $S^1$.  
We identify the quotient group $S^1/\Gamma_m$ with $S^1$ 
so as to consider the quotient space $S^3/\Gamma_m$ to be a principal $S^1$ bundle over $\CP^1$.  
The $m$-th Hirzebruch surface as $S^2$ bundle over $\CP^1$ 
is the fiber bundle $\pi_m: \Sigma_m\to\CP^1$ 
associated with the lens space $S^3/\Gamma_m$ with respect to the same effective action $S^1\lact S^2$ 
as in the $m=1$ case.  

\subsection{Effective actions $\U(2)/\Gamma_m\lact\Sigma_m$}\label{action}
On $S^3$, the unitary group 
$\U(2)=\{A=\left(\begin{smallmatrix}a&b\\c&d\end{smallmatrix}\right)\in\GL(2, \comp)\mid A{}^t\!\bar{A}=I\}$ 
acts effectively and transitively on the left by matrix multiplication $A.(z, w)=(az+bw, cz+dw)$.  
Let $\U(2)$ act trivially on $S^2$ and consider the diagonal action $\U(2)\lact S^3\times S^2$.  
Since this action commutes with the action $S^3\times S^2\ract S^1$ of structure group, 
it descends to an effective action $\U(2)\lact\Sigma_1$.  

When $m\ge 2$, by a slight abuse of notation, we also denote by $\Gamma_m$ the subgroup 
\[
\bigg\{\left(\begin{smallmatrix}e^{i\frac{2\pi l}{m}}&0\\0&e^{i\frac{2\pi l}{m}}\end{smallmatrix}\right)\bigg\lvert\ l\in\intg\bigg\}
\]
of $\U(2)$.  
The action $\U(2)\lact S^3$ descends to an effective action $\U(2)/\Gamma_m\lact S^3/\Gamma_m$, 
and $\U(2)/\Gamma_m$ thus acts effectively on $\Sigma_m$.  
We see later in Section \ref{odsub} that these actions are of cohomogeneity one.  

\subsection{Conformal K\"ahler triplets $\left(J_m(f), g_m(f), \omega_m(f)\right)$ on $\Sigma_m$}\label{cKahler}
We define a global moving frame $V$, $X$, $Y$ on $S^3$ by 
\begin{align}\label{ijk}
V: \begin{pmatrix}z\\w\end{pmatrix}\mapsto \begin{pmatrix}iz\\iw\end{pmatrix}, &&
X: \begin{pmatrix}z\\w\end{pmatrix}\mapsto \begin{pmatrix}-\overline{w}\\\overline{z}\end{pmatrix}, &&
Y: \begin{pmatrix}z\\w\end{pmatrix}\mapsto \begin{pmatrix}-i\overline{w}\\i\overline{z}\end{pmatrix}
\end{align}
so that $V$ is a $\U(2)$-invariant fundamental vector field 
corresponding to the vector field $\partial_{\theta}: e^{i\theta}\mapsto ie^{i\theta}$ on $S^1$, 
while $X$ and $Y$ are $\SU(2)$-invariant vector fields, which are not $\U(2)$-invariant.  
The tangential distribution $\widetilde{\mcal{H}}_1$ spanned by $X$ and $Y$ is a $\U(2)$-invariant principal connection on $S^3$.  
We denote by $\mcal{H}_1$ the induced connection on $\Sigma_1$, 
that is to say, 
$\mcal{H}_1$ is the image of the product distribution $\widetilde{\mcal{H}}_1\times\{0\}$ on $S^3\times S^2$ 
under the quotient map $q_1: S^3\times S^2\to\Sigma_1$.  
We see that $\mcal{H}_1$ is invariant under the action $\U(2)\lact\Sigma_1$.  

In order to be precise, we recall the definition of conformal K\"ahler triplets.  
Firstly, by a compatible triplet, we mean a triplet $(J, g, \omega)$ of 
an almost complex structure $J$, a Riemannian metric $g$ and an almost symplectic form $\omega$ 
satisfying the following conditions
\begin{align*}
g(JE, JF)=g(E, F), &&\omega(JE, JF)=\omega(E, F), &&\omega(E, F)=g(E, JF).  
\end{align*}
A compatible triplet $(J, g, \omega)$ is said to be (resp. conformal) K\"ahler 
if $J$ is integrable and $\omega$ is (resp. conformally) integrable.  
In oriented $2$-dimensional cases, 
one can start with an arbitrary Riemannian metric $g$, 
take its area form $\omega$, 
and then define $J$ by the equation $\omega(E, F)=g(E, JF)$ 
to obtain a compatible triplet, which has to be K\"ahler by dimension considerations.  

Let $\left(\check{J}, \check{g}, \check{\omega}\right)$ 
and $\left(\hat{J}, \hat{g}, \hat{\omega}\right)=\left(\hat{J}(f), \hat{g}(f), \hat{\omega}(f)\right)$ be the K\"ahler triplets 
defined by the Fubini-Study metric $\check{g}$ on $\CP^1$ and a $S^1$-invariant metric $\hat{g}=f^2(t)d\theta^2+dt^2$ on $S^2$, respectively.  
Using the invariance of $\left(\hat{J}, \hat{g}, \hat{\omega}\right)$ under the action $S^1\lact S^2$ of structure group, 
we deduce that there exists a unique compatible triplet $(J_1, g_1, \omega_1)=\left(J_1(f), g_1(f), \omega_1(f)\right)$ on $\Sigma_1$ 
with the following properties.  
\begin{enumerate}
\item The projection $\pi_1: (\Sigma_1, J_1, g_1, \omega_1)\to(\CP^1, \check{J}, \check{g}, \check{\omega})$ preserves the triplets.  
	More precisely, $\pi_1$ is both an almost holomorphic map and a Riemannian submersion.  
\item Each of the orthogonal and symplectic complements of the vertical distribution $\Ker d\pi\subset T\Sigma_1$ 
	coincides with the connection $\mcal{H}_1$.  
	In particular, $\mcal{H}_1$ is $J$-invariant.  
\item For each $\tilde{x}\in S^3$, 
	the map $\iota_{\tilde{x}}: S^2\to\Sigma_1$ defined by $\iota_{\tilde{x}}\left(\hat{x}\right)=q_1\left(\tilde{x}, \hat{x}\right)$ preserves the triplets 
	in the following sense: 
	the embedding $\iota_{\tilde{x}}$ is 
	an almost biholomorphic map, an isometry, and an almost symplectomorphism onto its image $\pi_1^{-1}\left(p_1\left(\tilde{x}\right)\right)$.  
\end{enumerate}
From the $\U(2)$-invariance 
of both the connection $\mcal{H}_1$ on $\Sigma_1$ and the compatible triplet $\left(\check{J}, \check{g}, \check{\omega}\right)$ on $\CP^1$, 
it follows that the conformal K\"ahler triplet $(J_1, g_1, \omega_1)$ is invariant under the action $\U(2)\lact\Sigma_1$.  
We also observe that the almost complex structure $J$ is integrable, 
giving $\Sigma_1$ the structure of locally trivial holomorphic fiber bundle, 
and $\omega$ is conformally integrable (see below for a description of its \textit{Lee form} \cite{Vai}).  

When $m\ge 2$, we define a $\U(2)/\Gamma_m$-invariant principal connection $\widetilde{\mcal{H}}_m$ on $S^3/\Gamma_m$ 
through covering map, 
i.e., $\widetilde{\mcal{H}}_m$ is defined as the image of $\widetilde{\mcal{H}}_1$ under the covering map $S^3\to S^3/\Gamma_m$.  
Following similar arguments out of the induced connection $\mcal{H}_m$ on $\Sigma_m$ 
and the same K\"ahler triplets $\left(\check{J}, \check{g}, \check{\omega}\right)$, $\left(\hat{J}, \hat{g}, \hat{\omega}\right)$ as in the $m=1$ case, 
we define a conformal K\"ahler triplet $(J_m, g_m, \omega_m)=\left(J_m(f), g_m(f), \omega_m(f)\right)$ on $\Sigma_m$ with analogous properties.  

We remark that the conformally symplectic form $\omega_m(f)$ is not integrable for each $m\ge 1$.  
This follows either from the explicit description $mf(t)dt$ of its Lee form 
on the submanifold $\left(S^3/\Gamma_m\right)\times(-T, T)$ to be defined in the next section, 
or from non-integrability\footnote{
	Under the assumption $m\ge 1$, the connection $\mcal{H}_m$ is not integrable 
	since the associated principal connection $\widetilde{\mcal{H}}_m$ corresponds to $m\check{\omega}$, 
	the area form $\check{\omega}$ of the Fubini-Study metric on $\CP^1$ multiplied by $m$ (cf. \cite{KobS}).  
}
of the connections $\mcal{H}_m$ (cf. \cite[4.1]{Wat}).  

\subsection{Open dense submanifolds $\left(S^3/\Gamma_m\right)\times(-T, T)$ in $\Sigma_m$}\label{odsub}
Let $S^1\times(-T, T)$ be the open dense cylinder embedded in $S^2$ and $\iota: S^1\times(-T, T)\to S^2$ the inclusion map.  
Since the cylinder is invariant under the $S^1$ action, 
the product map $\id\times\iota: S^3\times S^1\times(-T, T)\to S^3\times S^2$ induces an embedding 
of the associated fiber bundle $S^3\times_{S^1}\left(S^1\times(-T, T)\right)$ into $\Sigma_1$ with open dense image.  
On the other hand, $S^3\times_{S^1}\left(S^1\times(-T, T)\right)$ is canonically isomorphic to $S^3\times(-T, T)$.  
Hence it follows that $S^3\times(-T, T)$ is embedded in $\Sigma_1$ as an open dense submanifold.  
The situation is summarized as follows.  
Confer diagram \eqref{principalmap}.   
\[
\xymatrix{
S^3\times S^2\ar[rrrr]^-{\pr_1}\ar[dddd]_-{q_1}&&&&S^3\ar[dddd]^-{p_1}\\
&S^3\times S^1\times(-T, T)\ar[lu]_-{\id\times\iota}\ar[rd]\ar[rr]^-{\pr_1}\ar[dd]&&S^3\ar[dd]^-{p_1}\ar[ru]_-{\id}&\\
&&S^3\times(-T, T)\ar[rd]^-{p_1\circ\pr_1}\ar[ld]_-{\cong}&&\\
&S^3\times_{S^1}\left(S^1\times(-T, T)\right)\ar[rr]\ar[ld]&&\CP^1\ar[rd]^-{\id}&\\
\Sigma_1\ar[rrrr]^-{\pi_1}&&&&\CP^1.  
}
\]

All the structures defined previously on $\Sigma_1$ have simple descriptions on $S^3\times (-T, T)$.  
Firstly, the fiber bundle projection $\pi_1: \Sigma_1\to\CP^1$ is equivalent to 
the composition $p_1\circ\pr_1$ of the projection $\pr_1: S^3\times(-T, T)\to S^3$ onto the first factor and the Hopf fibration $p_1: S^3\to\CP^1$.  
Secondly, the $\U(2)$ action on $\Sigma_1$ is equivalent to the diagonal action $\U(2)\lact S^3\times(-T, T)$, 
where $\U(2)$ acts on the interval trivially.  
This explains why the action $\U(2)\lact\Sigma_1$ is of cohomogeneity one.  
Thirdly, the conformal K\"ahler triplet $(J_1, g_1, \omega_1)$ have the following expressions
\begin{gather*}
J_1=\begin{pmatrix}
0&1&0&0\\
-1&0&0&0\\
0&0&0&\frac{1}{f(t)}\\
0&0&-\frac{1}{f(t)}&0
\end{pmatrix}, \quad
g_1=\begin{pmatrix}
1&0&0&0\\
0&1&0&0\\
0&0&f^2(t)&0\\
0&0&0&1
\end{pmatrix}, \\
\omega_1=\begin{pmatrix}
0&1&0&0\\
-1&0&0&0\\
0&0&0&f(t)\\
0&0&-f(t)&0
\end{pmatrix}
\end{gather*}
in terms of the moving frame $X, Y, V, \partial_t$ on $S^3\times(-T, T)$.  
Here and henceforth, we use the same symbols $X, Y, V, \partial_t$ 
for the lifts of the original vector fields $X, Y, V$ on $S^3$ and $\partial_t$ on $(-T, T)$ through the canonical projections.  

When $m\ge 2$, 
the product manifold $\left(S^3/\Gamma_m\right)\times(-T, T)$ is embedded in $\Sigma_m$.  
The fiber bundle projection $\pi_m: \Sigma_m\to\CP^1$ and the action $\U(2)/\Gamma_m\lact\Sigma_m$ are described analogously.  
Additionally, we lift the conformal K\"ahler triplet $\left(J_m, g_m, \omega_m\right)$ 
through the covering map $S^3\times(-T, T)\to\left(S^3/\Gamma_m\right)\times (-T, T)$ 
to obtain their expressions 
\begin{gather*}\label{expressions_m}
J_m=\begin{pmatrix}
0&1&0&0\\
-1&0&0&0\\
0&0&0&\frac{1}{mf(t)}\\
0&0&-\frac{1}{mf(t)}&0
\end{pmatrix}, \quad
g_m=\begin{pmatrix}
1&0&0&0\\
0&1&0&0\\
0&0&m^2f^2(t)&0\\
0&0&0&1
\end{pmatrix}, \\
\omega_m=\begin{pmatrix}
0&1&0&0\\
-1&0&0&0\\
0&0&0&mf(t)\\
0&0&-mf(t)&0
\end{pmatrix}
\end{gather*}
in terms of $X, Y, V, \partial_t$.  
The factor $m$ comes in for the following reason.  
We remark that the structure group $S^1/\Gamma_m$ is identified with $S^1$ so that the diagram
\[
\xymatrix{
&S^1\ar[ld]\ar[rd]^-{e^{i\theta}\mapsto e^{im\theta}}&\\
S^1/\Gamma_m\ar[rr]^-{\cong}&&S^1
}
\]
commutes.  
The fundamental vector fields $V$ on $S^3$ and $V_m$ on $S^3/\Gamma_m$ 
corresponding to the same vector field $\partial_{\theta}$ on $S^1$ 
are thus related by the covering map up to factor $m$.  
That is, the pushforward of $V$ by the covering map happens to be $mV_m$.  
Since the vertical components of the conformal K\"ahler triplet $(J_m, g_m, \omega_m)$ with respect to $V_m, \partial_t$ 
are the same as the corresponding components of the triplet $(J_1, g_1, \omega_1)$ with respect to $V, \partial_t$, 
we have to take the factor $m$ into account as above.  

\subsection{Riemannian metrics satisfying conditions (I) and (II)}\label{characterization}
Let $\tilde{g}$ be a $\U(2)$-invariant Riemannian metric on $S^3$.  
Then, in terms of the moving frame $X$, $Y$, and $V$, $\tilde{g}$ is represented by a diagonal matrix 
$\left(\begin{smallmatrix}
b^2&0&0\\
0&b^2&0\\
0&0&a^2
\end{smallmatrix}\right)$ for some real numbers $a, b>0$.  
Indeed, 
since $\tilde{g}$ is $\SU(2)$-invariant, 
it is represented by a positive-definite symmetric matrix in terms of the $\SU(2)$-invariant moving frame $X, Y, V$; 
the isotropy subgroup at $(1, 0)\in S^3$ being the matrices of the form $\left(\begin{smallmatrix}1&0\\0&e^{i\theta}\end{smallmatrix}\right)$, 
it should be of the form above.  
A Riemannian manifold $S^3$ equipped with a $\U(2)$-invariant metric is called a Berger sphere (cf. \cite{Pet}).  

Therefore, if a metric $g$ on $\Sigma_1$ satisfies condition (II), 
then there exist strictly positive $C^{\infty}$ functions $f(t)$ and $h(t)$ on $(-T, T)$ such that 
$g$ is represented by the matrix 
$\left(\begin{smallmatrix}
h^2(t)&0&0&0\\
0&h^2(t)&0&0\\
0&0&f^2(t)&0\\
0&0&0&1
\end{smallmatrix}\right)$
on $S^3\times(-T, T)$ in terms of $X, Y, V, \partial_t$.  
If $g$ satisfies condition (I) as well, then $h(t)$ should be a constant, 
which is necessarily $1$ since the area of the base space $\CP^1$ is normalized to be $\pi$.  
Recalling the fact that $S^3\times(-T, T)$ is open and dense in $\Sigma_1$, 
it follows that, 
if $g$ satisfies conditions (I) and (II) at the same time, 
then $g$ agrees on the entire $\Sigma_1$ with the metric $g_1=g_1(f)$ defined previously.  

When $m\ge 2$, 
since each $\U(2)/\Gamma_m$-invariant metric on $S^3/\Gamma_m$ coincides with the metric of a Berger sphere up to covering, 
each metric $g$ satisfying conditions (I) and (II) is virtually represented on $S^3\times (-T, T)$ by
\begin{align}\label{Higher}
\begin{pmatrix}
1&0&0&0\\
0&1&0&0\\
0&0&m^2f^2(t)&0\\
0&0&0&1
\end{pmatrix}
\end{align}
in terms of $X, Y, V, \partial_t$.  

Lastly, we consider the boundary conditions for $f$ to be satisfied.  
Let $f: (-T, T)\to\real$ be a strictly positive function of class $C^{\infty}$ 
and define the Riemannian metric $g=g(f)$ on $S^3\times(-T, T)$ by the following matrix 
\begin{align}\label{gf}
g=g(f)\sim
\begin{pmatrix}
1&0&0&0\\
0&1&0&0\\
0&0&f^2(t)&0\\
0&0&0&1
\end{pmatrix}
\end{align}
in terms of $X, Y, V, \partial_t$.  
It is well known that a $S^1$-invariant metric $f^2(t)d\theta^2+dt^2$ defined on the open cylinder $S^1\times(-T, T)$ 
extends to a $C^{\infty}$ Riemannian metric on $S^2$ if and only if $f$ satisfies the following boundary conditions (cf. \cite[p. 213]{KW}, \cite[4.6]{Bes1})
\begin{align*}
f(\pm T)=0, &&f'(\pm T)=\mp 1, &&f^{(2l)}(\pm T)=0\quad\left(l=1, 2, \dots\right).  
\end{align*}
It follows that the metric $g(f)$ extends to a $C^{\infty}$ Riemannian metric on $\Sigma_1$ 
if and only if $f$ satisfies these boundary conditions.  

When $m\ge 2$, looking at the expression \eqref{Higher}, 
we notice that the metric $g(f)$ extends to a $C^{\infty}$ Riemannian metric on $\Sigma_m$ 
through the covering map $S^3\times(-T, T)\to \left(S^3/\Gamma_m\right)\times(-T, T)$ 
if and only if $f$ satisfies the following boundary conditions 
\begin{align}\label{BC}
f(\pm T)=0, &&f'(\pm T)=\mp m, &&f^{(2l)}(\pm T)=0\quad\left(l=1, 2, \dots\right).  
\end{align}

Summarizing this section, we obtain the following propositions.  
\begin{proposition}\label{identify}
For each $m\ge 1$, 
the collection of all $C^{\infty}$ metrics on the $m$-th Hirzebruch surface $\Sigma_m$ satisfying conditions \textup{(I)} and \textup{(II)} 
can be identified with all the metrics on $S^3\times(-T, T)$ of the form $g=g(f)$ in \eqref{gf}, 
where $f: (-T, T)\to\real$ runs over all the strictly positive $C^{\infty}$ functions satisfying boundary conditions \eqref{BC}.  
\end{proposition}

\begin{proposition}\label{consequences}
Each metric on $\Sigma_m$ satisfying conditions \textup{(I)} and \textup{(II)} is conformal K\"ahler 
with respect to the complex structure $J_m$.  
Moreover, the action $\U(2)/\Gamma_m\lact\Sigma_m$ of cohomogeneity one is by conformal K\"ahler automorphisms.  
The fiber bundle structure of $\Sigma_m$ is also compatible with the conformal K\"ahler triplet 
in the following sense\textup{:} 
the projection $\pi_m: \Sigma_m\to\CP^1$ is a conformal K\"ahler submersion \textup{(}cf. \textup{\cite{MR}}\textup{)} 
onto $\CP^1$ equipped with the Fubini-Study metric 
and every parallel translation between two fibers is a 
K\"ahler automorphism.  
In particular, $\pi_m$ is a Riemannian submersion with totally geodesic fibers \textup{(}cf. \textup{\cite{Herm}, \cite{Vil}}\textup{)}.  
\end{proposition}

\noindent Proposition \ref{identify} is due to B\'erard-Bergery \cite{Ber}.  
See also \cite[IV]{Bes2}, \cite[XV]{Bes2}, \cite[9.K]{Bes3}, \cite{HS}, \cite{MPPS}, and the references therein.  

\section{Tensor calculations}\label{Tensors}
In this section, we carry out tensor calculations on $S^3\times\opint$.  
Our main purpose is to obtain the following formulas.  
\begin{proposition}\label{TensorCalculations}
Consider the Riemannian metric $g=g(f)$ of \eqref{gf}, 
where we impose no boundary conditions for $f$.  
Then, 
its scalar curvature $R$ and squared tensor norm $\lvert W\rvert^2$ of Weyl tensor are written as
\begin{align}
R&=-2\frac{f''}{f}-2f^2+8, \label{scalar}\\
3\lvert W\rvert^2&=R^2-12f^2R+144(f')^2+36f^4\label{WeylSquared1}
\end{align}
on $S^3\times(-T, T)$.  
Furthermore, its Bach tensor $B=B_{ij}$ is diagonalized in terms of an orthonormal moving frame to be defined in \eqref{omf}, 
and the diagonal components are 
\begin{equation}\label{Bach1}
{\setlength\arraycolsep{1pt}
\begin{array}{rrrrrrrrrr}
24B_1=24B_2&=&2R''&+2\displaystyle\frac{f'}{f}R'&+R^2&-40f^2R&-16R&+96(f')^2&-276f^4&+576f^2, \\
24B_3&=&-4R''&&-R^2&+84f^2R&+16R&-96(f')^2&+492f^4&-1056f^2, \\
24B_4&=&&-4\displaystyle\frac{f'}{f}R'&-R^2&-4f^2R&+16R&-96(f')^2&+60f^4&-96f^2.  
\end{array}}
\end{equation}
\end{proposition}

\noindent Primes refer to derivatives with respect to $t\in (-T, T)$ unless stated otherwise.  

\subsection{First order derivatives}
We could regard $S^3$ itself as the Lie group consisting of unit quaternions.  
The vector fields $V$, $X$, and $Y$ on $S^3$ defined in \eqref{ijk} are then reinterpreted as the left-invariant vector fields 
corresponding respectively to the pure quaternions $i$, $j$, and $k$ of its Lie algebra, 
so that their Lie brackets have the following cyclic relationships
\begin{align}\label{Lie1}
[V, X]=2Y, &&[X, Y]=2V, &&[Y, V]=2X.  
\end{align}
From naturality of Lie brackets, we also have
\begin{align}\label{Lie2}
[V, \delt]=[X, \delt]=[Y, \delt]=0
\end{align}
apart from formulas \eqref{Lie1}, 
where $X$, $Y$, $V$, and $\delt$ are now understood to be vector fields on the product manifold $S^3\times\opint$ 
through the canonical projections.  
With respect to the metric $g=g(f)$, these vector fields $X$, $Y$, $V$, and $\delt$ form an orthogonal moving frame with corresponding norms
\begin{align}\label{InnerProduct}
\lvert X\rvert=\lvert Y\rvert=\lvert \delt\rvert=1, &&\lvert V\rvert=f(t).  
\end{align}

We recall that covariant derivatives are determined by inner products and Lie brackets.  
More precisely, the characteristic properties of Levi-Civita connection simply lead to the following formula 
(\textit{Koszul's formula} according to \cite{Pet})
\begin{equation}\label{Koszul}
\begin{split}
2\langle \nabla_EF, G\rangle
&=E\langle F, G\rangle+F\langle G, E\rangle-G\langle E, F\rangle\\
&\quad-\langle E, [F, G]\rangle+\langle F, [G, E]\rangle+\langle G, [E, F]\rangle, 
\end{split}
\end{equation}
valid for all vector fields $E$, $F$, and $G$ on a Riemannian manifold.  

Koszul's formula \eqref{Koszul} together with the previous formulas \eqref{Lie1}, \eqref{Lie2}, and \eqref{InnerProduct} 
allows us to calculate the first covariant derivatives in terms of the moving frame $X$, $Y$, $V$, $\delt$.  
The following two observations are fundamental\footnote{
	For its proof, we note that if $\phi=\phi(t)$ is a function on $S^3\times(-T, T)$ depending only on $t$, 
	then $X(\phi)=Y(\phi)=V(\phi)=0$ and $\delt(\phi)=\phi'(t)$.  
	}.  
\begin{enumerate}
\item The first three terms $E\langle F, G\rangle+F\langle G, E\rangle-G\langle E, F\rangle$ in \eqref{Koszul} vanish 
	unless $E$, $F$, $G$ agree with $V$, $V$, $\delt$ up to order.  
\item The last three terms $-\langle E, [F, G]\rangle+\langle F, [G, E]\rangle+\langle G, [E, F]\rangle$ in \eqref{Koszul} vanish 
	unless $E$, $F$, $G$ agree with $X$, $Y$, $V$ up to order.  
\end{enumerate}
An immediate consequence of these observations is that 
each of the covariant derivatives
\[
\nabla_XX, \quad\nabla_YY, \quad\nabla_{\delt}\delt, 
\quad\nabla_{\delt}X, \quad\nabla_X{\delt}, \quad\nabla_{\delt}Y, \quad\nabla_Y\delt
\]
is identically equal to zero.  
For the remaining cases, such computations as 
\begin{align*}
2\langle \nabla_VV, \delt\rangle&=-\delt\langle V, V\rangle=-2ff', \\*
\nabla_VV&=\langle \nabla_VV, \delt\rangle\delt=-ff'\delt\\
\intertext{and}
2\langle \nabla_XY, V\rangle&=-\langle X, [Y, V]\rangle+\langle Y, [V, X]\rangle+\langle V, [X, Y]\rangle=-2+2f^2+2=2f^2, \\*
\nabla_XY&=\langle \nabla_XY, \frac{1}{f}V\rangle\frac{1}{f}V=V
\end{align*}
lead to Table \ref{nablaTM} for the first covariant derivatives of $g$.  
\begin{table}\begin{center}\renewcommand{\arraystretch}{1.5}
	\begin{tabular}{c|cccc}
		$\nabla$&$X$&$Y$&$V$&$\delt$\\\hline
		$X$&$0$&$V$&$-f^2Y$&$0$\\
		$Y$&$-V$&$0$&$f^2X$&$0$\\
		$V$&$-(f^2-2)Y$&$(f^2-2)X$&$-ff'\delt$&$\frac{f'}{f}V$\\
		$\delt$&$0$&$0$&$\frac{f'}{f}V$&$0$
	\end{tabular}\vspace{.7\baselineskip}\caption{$\nabla$ in terms of $X$, $Y$, $V$, $\delt$}\label{nablaTM}
	\begin{tabular}{c|cccc}
		$\nabla$&$E_1$&$E_2$&$E_3$&$E_4$\\\hline
		$E_1$&$0$&$fE_3$&$-fE_2$&$0$\\
		$E_2$&$-fE_3$&$0$&$fE_1$&$0$\\
		$E_3$&$-\frac{f^2-2}{f}E_2$&$\frac{f^2-2}{f}E_1$&$-\frac{f'}{f}E_4$&$\frac{f'}{f}E_3$\\
		$E_4$&$0$&$0$&$0$&$0$
	\end{tabular}\vspace{.7\baselineskip}\caption{$\nabla$ in terms of $E_1, \dots ,E_4$}\label{nablaR}
\end{center}\end{table}

Henceforth, all tensor calculations are performed in terms of the orthonormal moving frame
\begin{equation}\label{omf}
E_1=X, \quad E_2=Y, \quad E_3=\frac{1}{f}V, \quad E_4=\delt.  
\end{equation}
The formulas for first derivatives obtained in this section are readily rewritten in terms of the new frame, 
and the results are summarized in Tables \ref{nablaR} and \ref{LieR}.  
\begin{table}\begin{center}\renewcommand{\arraystretch}{1.5}
	\begin{tabular}{c|cccc}
		$[\bullet, \ \bullet]$&$X$&$Y$&$V$&$\delt$\\\hline
		$X$&$0$&$2V$&$-2Y$&$0$\\
		$Y$&$-2V$&$0$&$2X$&$0$\\
		$V$&$2Y$&$-2X$&$0$&$0$\\
		$\delt$&$0$&$0$&$0$&$0$
	\end{tabular}\vspace{.7\baselineskip}\caption{Lie brackets of $X$, $Y$, $V$, $\delt$}\label{LieTM}
	\begin{tabular}{c|cccc}
		$[\bullet, \ \bullet]$&$E_1$&$E_2$&$E_3$&$E_4$\\\hline
		$E_1$&$0$&$2fE_3$&$-\frac{2}{f}E_2$&$0$\\
		$E_2$&$-2fE_3$&$0$&$\frac{2}{f}E_1$&$0$\\
		$E_3$&$\frac{2}{f}E_2$&$-\frac{2}{f}E_1$&$0$&$\frac{f'}{f}E_3$\\
		$E_4$&$0$&$0$&$-\frac{f'}{f}E_3$&$0$
	\end{tabular}\vspace{.7\baselineskip}\caption{Lie brackets of $E_1, \dots ,E_4$}\label{LieR}
\end{center}\end{table}
We observe that all components of the fourth row in Table \ref{nablaR} are zero, 
that is, $\nabla_{E_4}E_i=0$ for each $i=1, \dots, 4$.  
This observation simplifies the computations that follow.  

\subsection{Second order derivatives}
Let $\phi: S^3\times\opint\to\real$ be a smooth function depending only on $t$.  
The Hessian of $\phi$ is by definition written as
\begin{equation}\label{defHess}
\Hess\phi(E_i, E_j)=E_i(E_j\phi)-(\nabla_{E_i}E_j)\phi.  
\end{equation}
Since $E_1(\phi)=E_2(\phi)=E_3(\phi)=0$ and $E_4(\phi)=\phi'(t)$, 
it follows that the first term $E_i(E_j\phi)$ in \eqref{defHess} vanishes unless $i=j=4$, 
and we have $E_4(E_4\phi)=\phi''$ in this case.  
Additionally, Table \ref{nablaR} tells us that the second term $-(\nabla_{E_i}E_j)\phi$ in \eqref{defHess} vanishes 
unless $i=j=3$, and we have $-(\nabla_{E_3}E_3)\phi=\frac{f'}{f}\phi'$ in this case.  
Therefore, the Hessian of $\phi$ is diagonalized in terms of $\{E_i\}$, 
and its diagonal components are 
\begin{equation}\label{HessPhi}
\Hess_1\phi=\Hess_2\phi=0, \quad\Hess_3\phi=\frac{f'}{f}\phi', \quad\text{and}\quad\Hess_4\phi=\phi'', 
\end{equation}
whence we obtain 
\begin{equation}\label{LaplacePhi}
\Delta\phi=-\trace\Hess\phi=-\phi''-\frac{f'}{f}\phi'.  
\end{equation}  
Here, we have simply written, for instance, $\Hess_1\phi$ for the $(1, 1)$-component of $\Hess\phi$.  
In the sequel, this sort of notational conventions for covariant $2$-tensors is assumed without further mention.  

In order to compute the components of curvature tensor, 
we prepare the auxiliary tables 
for $\nabla_{E_i}\nabla_{E_j}E_k$ (Table \ref{SecondDerivatives1}) and $\nabla_{[E_i, E_j]}E_k$ (Table \ref{SecondDerivatives2}). 
These tables are consequences of direct computations using Tables \ref{nablaR} and \ref{LieR}.  

\begin{table*}\renewcommand{\arraystretch}{1.5}\centering
	\subfloat{
	\begin{tabular}{c|cccc}
		$\nabla_{\bullet}\nabla_{\bullet}E_1$&$E_1$&$E_2$&$E_3$&$E_4$\\\hline
		$E_1$&$0$&$f^2E_2$&$-(f^2-2)E_3$&$0$\\
		$E_2$&$0$&$-f^2E_1$&$0$&$0$\\
		$E_3$&$0$&$f'E_4$&$-\frac{(f^2-2)^2}{f^2}E_1$&$0$\\
		$E_4$&$0$&$-f'E_3$&$-\frac{f^2+2}{f}\frac{f'}{f}E_2$&$0$
	\end{tabular}\vspace{.7\baselineskip}}\hfill
	\subfloat{
	\begin{tabular}{c|cccc}
		$\nabla_{\bullet}\nabla_{\bullet}E_2$&$E_1$&$E_2$&$E_3$&$E_4$\\\hline
		$E_1$&$-f^2E_2$&$0$&$0$&$0$\\
		$E_2$&$f^2E_1$&$0$&$-(f^2-2)E_3$&$0$\\
		$E_3$&$-f'E_4$&$0$&$-\frac{(f^2-2)^2}{f^2}E_2$&$0$\\
		$E_4$&$f'E_3$&$0$&$\frac{f^2+2}{f}\frac{f'}{f}E_1$&$0$
	\end{tabular}\vspace{.7\baselineskip}}\hfill
	\subfloat{
	\begin{tabular}{c|cccc}
		$\nabla_{\bullet}\nabla_{\bullet}E_3$&$E_1$&$E_2$&$E_3$&$E_4$\\\hline
		$E_1$&$-f^2E_3$&$0$&$0$&$0$\\
		$E_2$&$0$&$-f^2E_3$&$0$&$0$\\
		$E_3$&$-(f^2-2)E_1$&$-(f^2-2)E_2$&$-\left(\frac{f'}{f}\right)^2E_3$&$0$\\
		$E_4$&$-f'E_2$&$f'E_1$&$-\frac{ff''-(f')^2}{f^2}E_4$&$0$
	\end{tabular}\vspace{.7\baselineskip}}\hfill
	\subfloat{
	\begin{tabular}{c|cccc}
		$\nabla_{\bullet}\nabla_{\bullet}E_4$&$E_1$&$E_2$&$E_3$&$E_4$\\\hline
		$E_1$&$0$&$0$&$-f'E_2$&$0$\\
		$E_2$&$0$&$0$&$f'E_1$&$0$\\
		$E_3$&$0$&$0$&$-\left(\frac{f'}{f}\right)^2E_4$&$0$\\
		$E_4$&$0$&$0$&$\frac{ff''-(f')^2}{f^2}E_3$&$0$
	\end{tabular}\vspace{.7\baselineskip}}
\caption{Second covariant derivatives $\nabla_{E_i}\nabla_{E_j}E_k$}\label{SecondDerivatives1}
\end{table*}

\begin{table*}\renewcommand{\arraystretch}{1.5}\centering
	\subfloat{
	\begin{tabular}{c|cccc}
		$\nabla_{[\bullet, \ \bullet]}E_1$&$E_1$&$E_2$&$E_3$&$E_4$\\\hline
		$E_1$&$0$&$-2(f^2-2)E_2$&$2E_3$&$0$\\
		$E_2$&$2(f^2-2)E_2$&$0$&$0$&$0$\\
		$E_3$&$-2E_3$&$0$&$0$&$-\frac{f^2-2}{f}\frac{f'}{f}E_2$\\
		$E_4$&$0$&$0$&$\frac{f^2-2}{f}\frac{f'}{f}E_2$&$0$
	\end{tabular}\vspace{.7\baselineskip}}\hfill
	\subfloat{
	\begin{tabular}{c|cccc}
		$\nabla_{[\bullet, \ \bullet]}E_2$&$E_1$&$E_2$&$E_3$&$E_4$\\\hline
		$E_1$&$0$&$2(f^2-2)E_1$&$0$&$0$\\
		$E_2$&$-2(f^2-2)E_1$&$0$&$2E_3$&$0$\\
		$E_3$&$0$&$-2E_3$&$0$&$\frac{f^2-2}{f}\frac{f'}{f}E_1$\\
		$E_4$&$0$&$0$&$-\frac{f^2-2}{f}\frac{f'}{f}E_1$&$0$
	\end{tabular}\vspace{.7\baselineskip}}\hfill
	\subfloat{
	\begin{tabular}{c|cccc}
		$\nabla_{[\bullet, \ \bullet]}E_3$&$E_1$&$E_2$&$E_3$&$E_4$\\\hline
		$E_1$&$0$&$-2f'E_4$&$-2E_1$&$0$\\
		$E_2$&$2f'E_4$&$0$&$-2E_2$&$0$\\
		$E_3$&$2E_1$&$2E_2$&$0$&$-\left(\frac{f'}{f}\right)^2E_4$\\
		$E_4$&$0$&$0$&$\left(\frac{f'}{f}\right)^2E_4$&$0$
	\end{tabular}\vspace{.7\baselineskip}}\hfill
	\subfloat{
	\begin{tabular}{c|cccc}
		$\nabla_{[\bullet, \ \bullet]}E_4$&$E_1$&$E_2$&$E_3$&$E_4$\\\hline
		$E_1$&$0$&$2f'E_3$&$0$&$0$\\
		$E_2$&$-2f'E_3$&$0$&$0$&$0$\\
		$E_3$&$0$&$0$&$0$&$\left(\frac{f'}{f}\right)^2E_3$\\
		$E_4$&$0$&$0$&$-\left(\frac{f'}{f}\right)^2E_3$&$0$
	\end{tabular}\vspace{.7\baselineskip}}
\caption{Second derivatives $\nabla_{[E_i, E_j]}E_k$}\label{SecondDerivatives2}
\end{table*}

\renewcommand{\proofname}{Proof of equation \eqref{scalar}.}
\begin{proof}
We successively compute various curvature quantities of second order, 
including those not necessary for the proof of equation \eqref{scalar} itself.  

Firstly, the components of curvature tensor
\begin{equation*}
R(E_i, E_j)E_k=\nabla_{E_i}\nabla_{E_j}E_k-\nabla_{E_j}\nabla_{E_i}E_k-\nabla_{[E_i, E_j]}E_k
\end{equation*}
are written in terms of the orthonormal moving frame $\{E_i\}$ as in Table \ref{CurvTensor}.  
Therefore, its squared tensor norm 
$\lvert R\rvert^2=\sum_{i, j, k=1}^4\lvert R(E_i, E_j)E_k\rvert^2$ 
is 
\begin{align}
\lvert R\rvert^2
&=4\left((3f^2-4)^2+f^4+6(f')^2\right)\notag\\*
	&\quad+2\left(2f^4+6(f')^2+\left(-\frac{f''}{f}\right)^2\right)+2\left(6(f')^2+\left(-\frac{f''}{f}\right)^2\right)\notag\\
&=4\left(-\frac{f''}{f}\right)^2+48(f')^2+44f^4-96f^2+64.  \label{Rsquared}
\end{align}

\begin{table*}\renewcommand{\arraystretch}{1.5}\centering
	\subfloat{
	\begin{tabular}{c|cccc}
		$R(\bullet, \ \bullet)E_1$&$E_1$&$E_2$&$E_3$&$E_4$\\\hline
		$E_1$&$0$&$(3f^2-4)E_2$&$-f^2E_3$&$0$\\
		$E_2$&$-(3f^2-4)E_2$&$0$&$-f'E_4$&$f'E_3$\\
		$E_3$&$f^2E_3$&$f'E_4$&$0$&$2f'E_2$\\
		$E_4$&$0$&$-f'E_3$&$-2f'E_2$&$0$
	\end{tabular}\vspace{.7\baselineskip}}\hfill
	\subfloat{
	\begin{tabular}{c|cccc}
		$R(\bullet, \ \bullet)E_2$&$E_1$&$E_2$&$E_3$&$E_4$\\\hline
		$E_1$&$0$&$-(3f^2-4)E_1$&$f'E_4$&$-f'E_3$\\
		$E_2$&$(3f^2-4)E_1$&$0$&$-f^2E_3$&$0$\\
		$E_3$&$-f'E_4$&$f^2E_3$&$0$&$-2fE_1$\\
		$E_4$&$f'E_3$&$0$&$2f'E_1$&$0$
	\end{tabular}\vspace{.7\baselineskip}}\hfill
	\subfloat{
	\begin{tabular}{c|cccc}
		$R(\bullet, \ \bullet)E_3$&$E_1$&$E_2$&$E_3$&$E_4$\\\hline
		$E_1$&$0$&$2f'E_4$&$f-2E_1$&$f'E_2$\\
		$E_2$&$-2f'E_4$&$0$&$f^2E_2$&$-f'E_1$\\
		$E_3$&$-f^2E_1$&$-f^2E_2$&$0$&$\frac{f''}{f}E_4$\\
		$E_4$&$-f'E_2$&$f'E_1$&$-\frac{f''}{f}E_4$&$0$
	\end{tabular}\vspace{.7\baselineskip}}\hfill
	\subfloat{
	\begin{tabular}{c|cccc}
		$R(\bullet, \ \bullet)E_4$&$E_1$&$E_2$&$E_3$&$E_4$\\\hline
		$E_1$&$0$&$-2f'E_3$&$-f'E_2$&$0$\\
		$E_2$&$2f'E_3$&$0$&$f'E_1$&$0$\\
		$E_3$&$f'E_2$&$-f'E_1$&$0$&$-\frac{f''}{f}E_3$\\
		$E_4$&$0$&$0$&$\frac{f''}{f}E_3$&$0$
	\end{tabular}\vspace{.7\baselineskip}}
\caption{Curvature tensor of type $(1, 3)$}\label{CurvTensor}
\end{table*}

Secondly, we observe from Table \ref{CurvTensor} that the components of Ricci tensor 
\begin{equation*}
\Ric(E_j, E_k)=\sum_{i=1}^4\langle R(E_i, E_j)E_k, E_i\rangle
\end{equation*}
vanishes unless $j=k$, 
and its diagonal components are 
\begin{equation}\label{Ric}
\Ric_1=\Ric_2=-2f^2+4, \quad\Ric_3=-\frac{f''}{f}+2f^2, \quad\text{and}\quad\Ric_4=-\frac{f''}{f}, 
\end{equation}
from which it follows that the squared tensor norm of $\Ric$ is 
\begin{equation}\label{RicSquared}
\lvert \Ric\rvert^2=2\left(-\frac{f''}{f}\right)^2-4ff''+12f^4-32f^2+32.  
\end{equation}

Thirdly, the desired equation \eqref{scalar} 
for scalar curvature is an immediate consequence of \eqref{Ric}.  
The following formulas 
\begin{align}
\lvert \tsRic\rvert^2&=\left(-\frac{f''}{f}\right)^2-8\left(-\frac{f''}{f}\right)-6ff''+11f^4+24f^2+16 \label{tsRicSquared}\\
R^2&=4\left(-\frac{f''}{f}\right)^2+32\left(-\frac{f''}{f}\right)+8ff''+4^4-32f^2+64, \label{scalarSquared}
\end{align}
later turn out to be convenient, 
where $\tsRic=\Ric-(R/4)g$ is the traceless Ricci tensor and $\lvert\tsRic\rvert$ represents its tensor norm.  
Note that $\lvert \tsRic\rvert^2=\lvert \Ric\rvert^2-R^2/4$ in our dimension.  
\end{proof}
\renewcommand{\proofname}{Proof}

From the formulas obtained so far, 
we prove the second formula of Proposition \ref{TensorCalculations}.  
\renewcommand{\proofname}{Proof of equation \eqref{WeylSquared1}.}
\begin{proof}
We denote by $\alpha\knp\beta$ the Kulkarni-Nomizu product of two symmetric $2$-tensors $\alpha$ and $\beta$.  
Its components are defined to be
\begin{equation*}
(\alpha\knp\beta)_{ijkl}=\alpha_{ik}\beta_{jl}+\alpha_{jl}\beta_{ik}-\alpha_{il}\beta_{jk}-\alpha_{jk}\beta_{il}.  
\end{equation*}
When $\beta$ happens to be the metric tensor $g$, 
a direct computation using indices shows that the tensor norms of $\alpha\knp g$ and $\alpha$ have the following relationship
\begin{equation*}
\lvert \alpha\knp g\rvert^2=4(n-2)\lvert \alpha\rvert^2+4(\trace\alpha)^2.  
\end{equation*}
Here, $n=4$ is the dimension of the manifold of our concern, and the norms and trace are taken with respect to $g$.  
Thus, in particular, we have 
\begin{align}\label{KNPandNORM2}
\lvert g\knp g\rvert^2=96, &&\lvert \tsRic\knp g\rvert^2=8\lvert \tsRic\rvert^2.  
\end{align}

Since the decomposition
\begin{equation*}
\Rm=W-\half\tsRic\knp g-\frac{R}{24}g\knp g
\end{equation*}
of covariant curvature tensor $\Rm$ is orthogonal, 
we have 
\begin{equation*}
\lvert W\rvert^2
=\lvert \Rm\rvert^2-\frac{1}{4}\lvert \tsRic\knp g\rvert^2-\frac{R^2}{24^2}\lvert g\knp g\rvert^2
=\lvert R\rvert^2-2\lvert \tsRic\rvert^2-\frac{1}{6}R^2.  
\end{equation*}
Therefore, equations \eqref{Rsquared}, \eqref{tsRicSquared}, and \eqref{scalarSquared} yield
\begin{equation}\label{WeylSquared2}
3\lvert W\rvert^2
=4\left(\frac{-f''}{f}\right)^2+32\left(\frac{-f''}{f}\right)+32ff''+144(f')^2+64f^4-128f^2+64.  
\end{equation}
We compare the following expression
\begin{equation*}
R^2-12f^2R
=4\left(\frac{-f''}{f}\right)^2+32\left(\frac{-f''}{f}\right)+32ff''+28f^4-128f^2+64
\end{equation*}
with the previous equation \eqref{WeylSquared2} to complete the proof.    
\end{proof}

\subsection{Third order derivatives}
We introduce the following notation 
\begin{equation}
\rho_i:=\frac{f^{(i)}}{f}, \qquad(i=1, 2, 3, 4)
\end{equation}
where $f^{(i)}$ stands for the $i$-th order derivative of $f=f(t)$.  
This is convenient because
\begin{equation}
\left( \rho_i\right)'=\rho_{i+1}-\rho_1\rho_i.  
\end{equation}
In the sequel, expressions involving higher derivatives of $f$ are written in terms of $\rho_1, \dots, \rho_4$.  

We shall compute the first covariant derivative $\nabla\Ric$ of Ricci curvature, 
whose components are by definition
\begin{equation*}
\nabla_{E_l}\Ric(E_i, E_j)
=E_l\left( \Ric(E_i, E_j)\right)-\Ric\left( \nabla_{E_l}E_i, E_j\right)-\Ric\left( E_i, \nabla_{E_l}E_j\right).  
\end{equation*}
It is readily observed that the first term $E_l\left( \Ric(E_i, E_j)\right)$ vanishes unless $l=4$, 
and $E_4\left( \Ric(E_i, E_j)\right)=\left( \Ric_{ij}\right)'$ in this case; 
on the other hand, the remaining terms vanish when $l=4$.  
The latter observation comes from $\nabla_{E_4}\equiv 0$ (cf. Table \ref{nablaR}).  
The components of $\nabla\Ric$ are summarized in Table \ref{nablaRic}.  

\begin{table*}\renewcommand{\arraystretch}{1.5}\centering
	\subfloat{
	\begin{tabular}{c|cccc}
		$\nabla_{E_1}\Ric\left( \bullet, \ \bullet\right)$&$E_1$&$E_2$&$E_3$&$E_4$\\\hline
		$E_1$&$0$&$0$&$0$&$0$\\
		$E_2$&$0$&$0$&$f\rho_2-4f^3+4f$&$0$\\
		$E_3$&$0$&$f\rho_2-4f^3+4f$&$0$&$0$\\
		$E_4$&$0$&$0$&$0$&$0$
	\end{tabular}\vspace{.7\baselineskip}}\hfill
	\subfloat{
	\begin{tabular}{c|cccc}
		$\nabla_{E_2}\Ric\left( \bullet, \ \bullet\right)$&$E_1$&$E_2$&$E_3$&$E_4$\\\hline
		$E_1$&$0$&$0$&$-f\rho_2+4f^3-4f$&$0$\\
		$E_2$&$0$&$0$&$0$&$0$\\
		$E_3$&$0$&$-f\rho_2+4f^3-4f$&$0$&$0$\\
		$E_4$&$0$&$0$&$0$&$0$
	\end{tabular}\vspace{.7\baselineskip}}\hfill
	\subfloat{
	\begin{tabular}{c|cccc}
		$\nabla_{E_3}\Ric\left( \bullet, \ \bullet\right)$&$E_1$&$E_2$&$E_3$&$E_4$\\\hline
		$E_1$&$0$&$0$&$0$&$0$\\
		$E_2$&$0$&$0$&$0$&$0$\\
		$E_3$&$0$&$0$&$0$&$-2f^2\rho_1$\\
		$E_4$&$0$&$0$&$-2f^2\rho_1$&$0$
	\end{tabular}\vspace{.7\baselineskip}}\hfill
	\subfloat{
	\begin{tabular}{c|cccc}
		$\nabla_{E_4}\Ric\left( \bullet, \ \bullet\right)$&$E_1$&$E_2$&$E_3$&$E_4$\\\hline
		$E_1$&$-4f^2\rho_1$&$0$&$0$&$0$\\
		$E_2$&$0$&$-4f^2\rho_1$&$0$&$0$\\
		$E_3$&$0$&$0$&$-\rho_3+\rho_1\rho_2+4f^2\rho_1$&$0$\\
		$E_4$&$0$&$0$&$0$&$-\rho_3+\rho_1\rho_2$
	\end{tabular}\vspace{.7\baselineskip}}
\caption{The first covariant derivative of Ricci curvature}\label{nablaRic}
\end{table*}

\subsection{Fourth order derivatives}
We are interested in the following two traces
\begin{equation}\label{defTwoTraces}
\nabla^p\nabla_p\Ric_{ij}, \qquad\nabla^p\nabla_j\Ric_{pi}
\end{equation}
of the second covariant derivative $\cd^2\Ric$ of Ricci curvature, 
in order to compute the Bach tensor with the help of Deridzi\'nski formula \eqref{Derdzinski}.  
We recall that the components of $\nabla^2\Ric$ is by definition written as
\begin{equation}\label{defNabla2Ric}
\begin{split}
\nabla_{E_k}\nabla_{E_l}\Ric(E_i, E_j)
&=E_k\left( \nabla_{E_l}\Ric(E_i, E_j)\right)-\cd_{\cd_{E_k}E_l}\Ric(E_i, E_j)\\
&\quad-\cd_{E_l}\Ric\left( \cd_{E_k}E_i, E_j\right)-\cd_{E_l}\Ric\left( E_i, \cd_{E_k}E_j\right).  
\end{split}
\end{equation}
For each term 
in \eqref{defNabla2Ric}, 
we determine the possibly nonzero components as in Tables \ref{0and*1}, \ref{0and*2}, \ref{0and*3}, and \ref{0and*4}, respectively.  
In each table, the $(i, j)$ component of the $(k, l)$ small matrix represents the corresponding term.  
For instance in Table \ref{0and*3}, the $(1, 2)$ component of the $(3, 3)$ small matrix represents 
$\cd_{E_3}\Ric\left( \cd_{E_3}E_1, E_2\right)$, which equals zero.  
Only the components filled out with $*$ are possibly nonzero, 
and we leave blank the components not of our interest to form the traces above.  

Tables \ref{0and*1}, \ref{0and*2}, \ref{0and*3}, and \ref{0and*4} tell us 
that both traces $\nabla^p\nabla_p\Ric_{ij}$ and $\nabla^p\nabla_j\Ric_{pi}$ are diagonalized in terms of $\{E_i\}$.  
Through direct computations using Tables \ref{nablaR} and \ref{nablaRic}, 
we obtain their diagonal components 
\begin{align}
\cd^p\cd_p\Ric_{11}
&=\cd^p\cd_p\Ric_{22}
=-6f^2\rho_2-8f^2\rho_1^2+8f^4-8f^2, \label{trace11}\\
\cd^p\cd_p\Ric_{33}
&=-\left( \rho_2\right)''-\rho_1\left(\rho_2\right)'+8f^2\rho_2+4f^2\rho_1^2-16f^4+16f^2, \label{trace12}\\
\cd^p\cd_p\Ric_{44}
&=-\left( \rho_2\right)''-\rho_1\left(\rho_2\right)'+4f^2\rho_1^2\label{trace13}
\intertext{and}
\cd^p\cd_1\Ric_{p1}
&=\cd^p\cd_2\Ric_{p2}
=f^2\rho_2-f^4+4f^2, \label{trace21}\\
\cd^p\cd_3\Ric_{p3}
&=-\rho_1\left( \rho_2\right)'-4f^2\rho_2-2f^2\rho_1^2+8f^4-8f^2, \label{trace22}\\
\cd^p\cd_4\Ric_{p4}
&=-\left( \rho_2\right)''-2f^2\rho_1^2.  \label{trace23}
\end{align}

\begin{table}\begin{center}\setlength{\tabcolsep}{.3em}
	\begin{minipage}{.49\linewidth}\centering
	\begin{tabular}{|cccc|cccc|cccc|cccc|}
		\hline
		0&0&0&0&0&0&0&0&0&0&0&0&0&0&0&0\\
		0&0&0&0&&&&&&&&&&&&\\
		0&0&0&0&&&&&&&&&&&&\\
		0&0&0&0&&&&&&&&&&&&\\\hline
		&&&&0&0&0&0&&&&&&&&\\
		0&0&0&0&0&0&0&0&0&0&0&0&0&0&0&0\\
		&&&&0&0&0&0&&&&&&&&\\
		&&&&0&0&0&0&&&&&&&&\\\hline
		&&&&&&&&0&0&0&0&&&&\\
		&&&&&&&&0&0&0&0&&&&\\
		0&0&0&0&0&0&0&0&0&0&0&0&0&0&0&0\\
		&&&&&&&&0&0&0&0&&&&\\\hline
		&&&&&&&&&&&&0&0&0&0\\
		&&&&&&&&&&&&0&0&0&0\\
		&&&&&&&&&&&&0&0&0&0\\
		0&0&0&0&0&0&0&0&0&0&$*$&0&0&0&0&$*$\\\hline
	\end{tabular}\vspace{.7\baselineskip}\caption{$E_k\left( \nabla_{E_l}\Ric\left( E_i, E_j\right)\right)$}\label{0and*1}\end{minipage}
	\begin{minipage}{.49\linewidth}\centering
	\begin{tabular}{|cccc|cccc|cccc|cccc|}
		\hline
		0&0&0&0&0&0&0&0&0&0&$*$&0&0&0&0&0\\
		0&0&0&0&&&&&&&&&&&&\\
		0&0&0&0&&&&&&&&&&&&\\
		0&0&0&0&&&&&&&&&&&&\\\hline
		&&&&0&0&0&0&&&&&&&&\\
		0&0&0&0&0&0&0&0&0&0&$*$&0&0&0&0&0\\
		&&&&0&0&0&0&&&&&&&&\\
		&&&&0&0&0&0&&&&&&&&\\\hline
		&&&&&&&&$*$&0&0&0&&&&\\
		&&&&&&&&0&$*$&0&0&&&&\\
		$*$&0&0&0&0&$*$&0&0&0&0&$*$&0&0&0&0&$*$\\
		&&&&&&&&0&0&0&$*$&&&&\\\hline
		&&&&&&&&&&&&0&0&0&0\\
		&&&&&&&&&&&&0&0&0&0\\
		&&&&&&&&&&&&0&0&0&0\\
		0&0&0&0&0&0&0&0&0&0&0&0&0&0&0&0\\\hline
	\end{tabular}\vspace{.7\baselineskip}\caption{$\nabla_{\nabla_{E_k}E_l}\Ric(E_i, E_j)$}\label{0and*2}\end{minipage}
	\begin{minipage}{.49\linewidth}\centering
	\begin{tabular}{|cccc|cccc|cccc|cccc|}
		\hline
		0&0&0&0&0&0&0&0&0&0&0&0&0&0&0&0\\
		0&$*$&0&0&&&&&&&&&&&&\\
		0&0&$*$&0&&&&&&&&&&&&\\
		0&0&0&0&&&&&&&&&&&&\\\hline
		&&&&$*$&0&0&0&&&&&&&&\\
		0&0&0&0&0&0&0&0&0&0&0&0&0&0&0&0\\
		&&&&0&0&$*$&0&&&&&&&&\\
		&&&&0&0&0&0&&&&&&&&\\\hline
		&&&&&&&&0&0&0&0&&&&\\
		&&&&&&&&0&0&0&0&&&&\\
		0&0&0&0&0&0&0&0&0&0&$*$&0&0&0&0&$*$\\
		&&&&&&&&0&0&0&$*$&&&&\\\hline
		&&&&&&&&&&&&0&0&0&0\\
		&&&&&&&&&&&&0&0&0&0\\
		&&&&&&&&&&&&0&0&0&0\\
		0&0&0&0&0&0&0&0&0&0&0&0&0&0&0&0\\\hline
	\end{tabular}\vspace{.7\baselineskip}\caption{$\nabla_{E_l}\Ric\left( \nabla_{E_k}E_i, E_j\right)$}\label{0and*3}\end{minipage}
	\begin{minipage}{.49\linewidth}\centering
	\begin{tabular}{|cccc|cccc|cccc|cccc|}
		\hline
		0&0&0&0&0&$*$&0&0&0&0&0&0&0&0&0&0\\
		0&$*$&0&0&&&&&&&&&&&&\\
		0&0&$*$&0&&&&&&&&&&&&\\
		0&0&0&0&&&&&&&&&&&&\\\hline
		&&&&$*$&0&0&0&&&&&&&&\\
		$*$&0&0&0&0&0&0&0&0&0&0&0&0&0&0&0\\
		&&&&0&0&$*$&0&&&&&&&&\\
		&&&&0&0&0&0&&&&&&&&\\\hline
		&&&&&&&&0&0&0&0&&&&\\
		&&&&&&&&0&0&0&0&&&&\\
		$*$&0&0&0&0&$*$&0&0&0&0&$*$&0&0&0&0&$*$\\
		&&&&&&&&0&0&0&$*$&&&&\\\hline
		&&&&&&&&&&&&0&0&0&0\\
		&&&&&&&&&&&&0&0&0&0\\
		&&&&&&&&&&&&0&0&0&0\\
		0&0&0&0&0&0&0&0&0&0&0&0&0&0&0&0\\\hline
	\end{tabular}\vspace{.7\baselineskip}\caption{$\nabla_{E_l}\Ric\left( E_i, \nabla_{E_k}E_j\right)$}\label{0and*4}\end{minipage}
\end{center}\end{table}

We are now in a position to compute the Bach tensor of $g=g(f)$ using Derdzi\'nski formula (\cite[Equation (24)]{Der})
\begin{equation}\label{Derdzinski}
\begin{split}
B_{ij}
&=\cd^p\cd_j\Ric_{pi}-\half\cd^p\cd_p\Ric_{ij}-\frac{1}{3}\Hess_{ij}R-\frac{1}{12}(\Delta R)g_{ij}\\
&\quad+\frac{1}{3}R\cdot\Ric_{ij}-\Ric_{i}^p\Ric_{pj}+\frac{1}{12}\left( 3\lvert \Ric\rvert^2-R^2\right)g_{ij}.  
\end{split}
\end{equation}

\renewcommand{\proofname}{Proof of equations \eqref{Bach1}.}
\begin{proof}
First of all, since it follows from the previous calculations that the Bach tensor is diagonalized, 
we have only to compute its diagonal components.  
The first four terms in formula \eqref{Derdzinski} involving fourth derivatives are written respectively in terms of $f$ using equations 
\eqref{trace11}-\eqref{trace13}, \eqref{trace21}-\eqref{trace23}, \eqref{HessPhi}, and \eqref{LaplacePhi}.  
The last three terms are written in terms of $f$ using equations \eqref{scalar}, \eqref{Ric}, \eqref{RicSquared}, and \eqref{scalarSquared}.  
Hence, we have all of its diagonal components
\begin{equation}\label{Bach}
\renewcommand{\arraystretch}{2}{\setlength\arraycolsep{1pt}
\begin{array}{rrrrrrrrrrr}
6B_1=6B_2&=&-\rho_4&+\rho_1\rho_3&+2\rho_2^2&-1\rho_1^2\rho_2&+20f^2\rho_2&+20f^2\rho_1^2&-48f^4&+64f^2&-11, \\
6B_3&=&2\rho_4&-4\rho_1\rho_3&-3\rho_2^2&+4\rho_1^2\rho_2&-40f^2\rho_2&-20f^2\rho_1^2&+80f^4&-96f^2&+11, \\
6B_4&=&&2\rho_1\rho_3&-\rho_2^2&-2\rho_1^2\rho_2&-20f^2\rho_1^2&&+16f^4&-32f^2&+11.  
\end{array}}
\end{equation}
Rewriting these in terms of the scalar curvature $R$, we obtain equations \eqref{Bach1}.  
\end{proof}
\renewcommand{\proofname}{Proof}

We close this section with listing a couple of observations to support our computational results.  
\begin{enumerate}
\item When we impose boundary conditions \eqref{BC} on $f$, 
	then Chern-Gauss-Bonnet formula on $\Sigma_m$ is reduced to the fundamental theorem of calculus.  
\item The diagonalized tensor $B$ defined as in \eqref{Bach1} is indeed trace-free and divergence-free.  
	That is, $B_1+B_2+B_3+B_4=0$ and $\frac{d}{dt}B_4-\frac{f'}{f}(B_3-B_4)=0$.  
\item For each fixed $m$, consider the functional $2\pi^2/m\int_{-T}^Tf\lvert W\rvert^2dt$ 
	with domain all the functions $f$ satisfying boundary conditions \eqref{BC}.  
	Then, its Euler-Lagrange equation is precisely $B_3=0$.  
\end{enumerate}
Note that the functional $2\pi^2/m\int_{-T}^Tf\lvert W\rvert^2dt$ in (3)
is equivalent to the Weyl functional $\int_{\Sigma_m}\lvert W\rvert^2\dVol$
restricted to the metrics on $\Sigma_m$ with $\U(2)/\Gamma_m$-symmetry (cf. Section \ref{FurtherProperties}).  
Therefore, the last observation says that the metric $g=g(f)$ defines a critical metric on $\Sigma_m$ with respect to the restricted Weyl functional 
if and only if $B_3=0$.  
The observations (1), (2), and (3) seem to be more than coincidence.  

\section{Metrics on $\Sigma_m$ with constant scalar curvature}\label{CscMetrics}
\subsection{Proofs of Theorems \ref{cscExistence} and \ref{nonExistence}}\label{MainProofs}
Proposition \ref{identify} and equation \eqref{scalar} 
reduce the construction of constant scalar curvature metrics on $\Sigma_m$ 
to solving the ordinary differential equation
\begin{equation}\label{DuffingEq.}
f''=-f^3-\frac{R-8}{2}
\end{equation}
under boundary conditions \eqref{BC} and the positivity assumption $f(t)>0$ for $t\in (-T, T)$.  

The following lemma states that our free-boundary value problem has unique solutions.   
\begin{lemma}\label{Duffing}
Let $m$ be a positive integer and $\beta$ a real number.  
Then, there exists a unique real number $T>0$ and a unique $C^{\infty}$ function $f: [-T, T]\to\real$ 
satisfying the following conditions.  
\begin{itemize}\renewcommand{\labelitemi}{$\bullet$}
\item $f$ solves the ordinary differential equation $f''=-f^3+\beta f$ on $[-T, T]$.  
\item $f$ satisfies the boundary conditions $f(\pm T)=0, f'(\pm T)=\mp m$.  
\item $f$ is strictly positive on $(-T, T)$.  
\end{itemize}
Moreover, this function $f$ has the following additional properties.  
\begin{itemize}\renewcommand{\labelitemi}{$\bullet$}
\item $f$ satisfies the boundary conditions $f^{(2l)}(\pm T)=0$ for $l=1, 2, \dots$ .  
\item $f$ is an even function.  
\item $f$ is strictly increasing on $[-T, 0]$ and strictly decreasing on $[0, T]$.    
\end{itemize}
\end{lemma}
\begin{proof}
A slightly formal computation
\[
\frac{df}{df'}=\frac{df}{dt}\bigg/\frac{df'}{dt}=\frac{f'}{-f^3+\beta f}
\]
leads to the integral 
\begin{equation}\label{FI}
2(f')^2=-f^4+2\beta f^2+2m^2
\end{equation} 
of this boundary value problem.  
The rest of the proof is similar to the arguments for defining trigonometric functions and Jacobian elliptic functions 
through ordinary differential equations.  
\end{proof}

We now prove our main results.  
\renewcommand{\proofname}{Proofs of Theorems \ref{cscExistence} and \ref{nonExistence}. }
\begin{proof}
The existence part of Lemma \ref{Duffing} yields Theorem \ref{cscExistence}.  
More precisely, for each integer $m\ge 1$ and each real number $R$, 
we define a constant scalar curvature metric $g_m(R)$ on $\Sigma_m$ as follows.  
Firstly, take a function $f$ to be the positive solution of our boundary value problem \eqref{BC} and \eqref{DuffingEq.}.  
Secondly, define the metric $g=g(f)$ on $S^3\times (-T, T)$ through the matrix \eqref{gf}.  
Thirdly, define a Riemannian metric on the open dense submanifold of $\Sigma_m$ 
so that each of the covering map $S^3\times(-T, T)\to\left(S^3/\Gamma_m\right)\times(-T, T)$ 
and embedding $\left(S^3/\Gamma_m\right)\times(-T, T)\to\Sigma_m$ be a local isometry (cf. \S \ref{odsub}).  
Arguments in Section \ref{Preliminaries} indicate that the metric so defined 
extends smoothly to a metric on the whole $\Sigma_m$.  
This smooth metric is our $g_m(R)$, 
which has constant scalar curvature $R$ on $\Sigma_m$
since its scalar curvature is identically equal to the constant $R$ on the open dense submanifold.  

We prove Theorem \ref{nonExistence} as follows.  
Proposition \ref{identify} and the uniqueness part of Lemma \ref{Duffing} tell us that 
the metrics $g_m(R)$ of Theorem \ref{cscExistence} are the only constant scalar curvature metrics on $\Sigma_m$ satisfying conditions (I) and (II).  
On the other hand, 
we observe that the scalar curvature $R$ of $g=g(f)$ cannot be constant 
in order for a function $f$ to solve simultaneous equations \eqref{Bach1} with $B_1=B_2=B_3=B_4=0$.  
The metrics $g_m(R)$'s are therefore not Bach flat.  
Since critical metrics with respect to linear or quadratic curvature functionals have constant scalar curvature 
unless it is Bach flat (cf. \cite[4.H]{Bes3}, \cite{GV}), 
Theorem \ref{nonExistence} follows.  
\end{proof}
\renewcommand{\proofname}{Proof}

\subsection{Further properties of $g_m(R)$}\label{FurtherProperties}
We look at the metrics $g_m(R)$ with constant scalar curvature $R$ in more detail.  
One of the ingredients required for their analysis is the following.  
\begin{lemma}\label{DefiniteIntegrals}
Let $f$ be the function of Lemma \ref{Duffing}.  Then, we have
\begin{align}
T&=\frac{1}{\sqrt[4]{2m^2+\beta^2}}K(k), \\\label{INTf1}
\int_{-T}^Tf(t)dt&=2\sqrt{2}\Arcsin(k), \\\label{INTf3}
\int_{-T}^Tf^3(t)dt&=2\beta\sqrt{2}\Arcsin(k)+2m, \\\label{INTf5}
\int_{-T}^Tf^5(t)dt&=\left(2m^2+3\beta^2\right)\sqrt{2}\Arcsin(k)+3m\beta.   
\end{align}
Here, the complete elliptic integral $K(k)$ of the first kind and its modulus $k>0$ are defined to be
\begin{align}\label{Kk}
K(k)=\int_0^1\frac{dx}{\sqrt{1-x^2}\sqrt{1-k^2x^2}}, &&
k^2=\half\left(1+\frac{\beta}{\sqrt{2m^2+\beta^2}}\right).  
\end{align}
\end{lemma}
\begin{proof}
Direct calculations using integral \eqref{FI}.  
Additional properties of $f$ in Lemma \ref{Duffing} help.  
\end{proof}

In the sequel, $\beta$ denotes the constant $-(R-8)/2$, 
and we promise that $k>0$ refers to the constant determined by this $\beta$ as in the second equation of \eqref{Kk}.  
\begin{proposition}[Behavior of linear curvature functionals]\label{linear}
The Yamabe functional $Y$, or equivalently the Einstein-Hilbert functional $E$, 
attains the following value
\begin{equation}\label{Yamabe}
Y\left(g_m(R)\right)=R\sqrt{\Vol(g_m(R))}
=2\sqrt[4]{2}\pi R\sqrt{\frac{\Arcsin(k)}{m}}
\end{equation}
at $g_m(R)$.  
The metric $g_m(R)$ is critical with respect to $Y$ by its definition, but not critical with respect to $E$.  

There exists a real number $\epsilon>0$ depending on $m$ 
so that, if $R$ is less than $\epsilon$, 
then $g_m(R)$ is a unique Yamabe minimizer in its conformal class up to homothety.  
On the other hand, if $R$ is greater than 24 \textup{(}regardless of $m$\textup{)}, 
then $g_m(R)$ is not stable with respect to $Y$ and hence not a Yamabe minimizer.  
\end{proposition}

\begin{proof}
Let $\tilde{g}$ be the constant curvature metric on $S^3$ of radius $1$ 
and $\Omega$ its volume form.  
We note that the volume form of $g=g(f)$ is equal to $f(t)\Omega\wedge dt$.  
Thus, the integral of a $\U(2)$-invariant function $\phi=\phi(t)$ on $S^3\times(-T, T)$ is
\begin{equation}\label{integral}
\int_{S^3\times(-T, T)}\phi\dVol_{g(f)}=2\pi^2\int_{-T}^Tf(t)\phi(t)dt, 
\end{equation}
where $2\pi^2$ is the volume of $\tilde{g}$.  
In particular, the volume of $g(f)$ is equal to $2\pi^2\int_{-T}^Tf(t)dt$.  
Therefore, using \eqref{INTf1}, we obtain 
\[
\Vol(g_m(R))=\frac{1}{m}\cdot 2\pi^2\int_{-T}^Tf(t)dt=4\sqrt{2}\pi^2\frac{\Arcsin(k)}{m}, 
\]
which proves \eqref{Yamabe}.  
We should be aware of the factor $m$ appearing through the covering map.  
The critical points of $Y$ and $E$ are constant scalar curvature metrics and Einstein metrics, respectively, 
and we observe from \eqref{Ric} that $g_m(R)$ is not Einstein.  

A slight modification to Theorem 5.1 of B\"ohm, Wang, and Ziller \cite{BWZ} ensures the existence of such an $\epsilon>0$.  
We remark that, according to de Lima, Piccione, and Zedda \cite{LPZ}, 
$g_m(R)$ is a unique constant scalar curvature metric in its conformal class up to homothety.  

We recall that a constant scalar curvature metric $g$ on a closed $4$-dimensional manifold is stable with respect to $Y$ if and only if 
its scalar curvature $R$ and first eigenvalue $\lambda_1>0$ of Laplacian satisfies 
$\lambda_1\ge R/3$ (cf. \cite{KobO}).  
Thus, for the last assertion, we have only to estimate $\lambda_1$ from above.  
For the record, we also derive a lower bound of the first eigenvalue in what follows .  

Since $\pi_m: \left(\Sigma_m, g_m(R)\right) \to\left(\CP^1, \check{g}\right)$ 
is a Riemannian submersion with totally geodesic fibers onto the Fubini-Study metric (Proposition \ref{consequences}), 
the first eigenvalues 
\begin{align*}
\lambda_1=\lambda_1\left(\Sigma_m, g_m(R)\right),&& 
\check{\lambda}_1=\check{\lambda}_1\left(\CP^1, \check{g}\right)=8,&&
\hat{\lambda}_1=\hat{\lambda}_1\left(S^2, f^2(t)d\theta^2+dt^2\right)
\end{align*}
of the total space, base space, and fiber satisfy the following inequalities
\begin{align}\label{eigen1}
\min\{8, \hat{\lambda}_1\}\le\lambda_1, &&\lambda_1\le 8
\end{align}
(cf. \cite{Bor} for the first and \cite{BB} for the second).  
The second inequality of \eqref{eigen1} completes the proof of Proposition \ref{linear}.  
On the other hand, Cheeger's isoperimetric inequality \cite{Che} and Hersch's inequality \cite{Hers} yield 
$h^2/4\le\hat{\lambda}_1\le 8\pi/a$, 
where $h$ and $a$ are respectively the isoperimetric constant and area of the $S^1$-invariant metric $f^2(t)d\theta^2+dt^2$.  
Furthermore, results of Ritor\'e \cite{Rit} show, 
via Yau's argument \cite[p. 489]{Yau}, 
that the value $h$ is attained by a domain of area $a/2$ whose boundary is a \textit{nodoid} of length $4T$.  
We note that the differential equation \eqref{DuffingEq.} and monotonicity of $f$ in Lemma \ref{Duffing} 
imply that the Gaussian curvature $-f''/f$ of $f^2(t)d\theta^2+dt^2$ is monotone.  
Hence, $h$ is equal to $8T/a$.  
Since $T$ and $a=2\pi\int_{-T}^Tf(t)dt$ can be written in terms of $m$ and $R$ (Lemma \ref{DefiniteIntegrals}), 
we obtain
\begin{equation}\label{eigen2}
\frac{1}{2\pi^2\sqrt{2m^2+\beta^2}}\frac{K^2(k)}{\Arcsin^2(k)}\le\hat{\lambda}_1\le\frac{\sqrt{2}}{\Arcsin(k)}.  
\end{equation}
Stability of some metrics $g_m(R)$ follows from inequalities \eqref{eigen1} and \eqref{eigen2}.  
\end{proof}

\begin{proposition}[Behavior of quadratic curvature functionals]
The $\mcal{B}_t$-functional attains the following value
\begin{equation*}
\mcal{B}_t\left(g_m(R)\right)
=\frac{2\pi^2}{m}\left(72m^2+\frac{6t+59}{3}R^2-272R+960\right)\sqrt{2}\Arcsin(k)
-4\pi^2(19R-120)
\end{equation*}
at $g_m(R)$.  
The metric $g_m(R)$ is not critical with respect to $\mcal{B}_t$-functional for each $m\ge 1$, each $R\in\real$, and each $t\in\real$.  
\end{proposition}

\begin{proof}
The first part follows from direct calculations using \eqref{WeylSquared1}, \eqref{scalarSquared}, \eqref{integral}, and Lemma \ref{DefiniteIntegrals}.  
A $B^t$-flat metric, which is by definition a critical point of $\mcal{B}_t$-functional, 
is either Bach flat or 
a constant scalar curvature metric whose Bach tensor $B$ is a constant multiple of its traceless Ricci tensor (cf. \cite{GV}).  
We have already observed that $g_m(R)$ is not Bach flat, 
and it follows from equations \eqref{Ric} and \eqref{Bach} that the latter situation does not occur neither.  
\end{proof}

\section{The Bach-flat equation}\label{BachFlat}
Equations \eqref{Bach} are precisely the system of ordinary differential equations describing Bach-flat metrics satisfying conditions (I) and (II).  
We remark that, 
from the conformal invariant property of Bach tensor and a B\'erard-Bergery's result on cohomogeneity-one Riemannian geometry \cite[\S 7]{Ber}, 
it follows that these equations describe Bach-flat metrics on Hirzebruch surfaces with $4$-dimensional isometry group 
(a priori without any relation to our fixed action $\U(2)/\Gamma_m\lact\Sigma_m$).  

We slightly simplify the system of ODEs.  
From the trace-free and divergence-free conditions $B_1+B_2+B_3+B_4=0$ and $\frac{d}{dt}B_4=\frac{f'}{f}(B_3-B_4)$, 
we observe: 
For a strictly positive $C^{\infty}$ function $f: (-T, T)\to\real$ satisfying boundary conditions \eqref{BC}, 
the following statements (A) and (B) are equivalent.  
(A) $f$ is a solution to the system of ODEs \eqref{Bach}.  
(B) $f$ satisfies $B_4=0$, 
and the regular point set $\{t\in[-T, T]\mid f'(t)\neq 0\}$ of $f$ is open and dense in $[-T, T]$.  
Therefore, the single ordinary differential equation $B_4=0$ becomes of our interest.  
In view of this observation, 
we perform the change of variables $x=f^2$, $y=(f')^2$ to transform the equation $B_4=0$ into 
\begin{align}\label{BachFlatODE}
4yy''-(y')^2=20y-16x^2+32x-11.  
\end{align}
Here, $y'$ and $y''$ represent the first and second derivatives of $y$ with respect to $x$.  

Such mathematicians as B\'erard-Bergery look for generalization of Page metric 
and verify that there exists no Einstein metric 
on higher Hirzebruch surfaces $\Sigma_m$ with $4$-dimensional isometry group ($m\ge 2$).  
See \cite{Ber} and \cite[9.K]{Bes3}.  
The author thinks it interesting to consider the analogous problem for Bach-flat metrics.  
However, equation \eqref{BachFlatODE}, 
which does have movable essential singularities according to \cite[XIV]{Inc}, 
is seemingly not easy to solve\footnote{
	Professor Shimomura has identified that the equation $B_4=0$ has two particular solutions defined respectively by 
	\begin{align*}
	(f')^2=f^4-2f^2+7/12, &&
	(f')^2=4f^4-8f^2+53/12.  
	\end{align*}
	Such functions do not define Bach-flat metrics on any Hirzebruch surface since neither $7/12$ nor $7/53$ is an integer.  
	Compare \eqref{FI}.  }.  

\subsection*{Acknowledgements}
A part of this paper is based on \cite{Oto}, 
written under the supervision of Professor H. Izeki.  
I express my gratitude to Professor O. Kobayashi for his genuine advice and encouragement 
after reading an earlier version of the thesis.  
During his stay in Japan, Professor J. Viaclovsky kindly recommended \cite{Ber} and \cite{Der}, 
thereby reminding me 
that the constant scalar curvature metrics constructed in my thesis on $\Sone$ 
could be generalized to higher Hirzebruch surfaces, 
and that the corresponding $\mcal{B}_t$-flat analogue should be in presence.  
I am also grateful to Professor S. Shimomura for improving my understanding of the Bach-flat equation, 
and to Professor S. Matsuo for his advice on an earlier version of this paper.  
This work was supported by Japan Society for the Promotion of Science under Research Fellowship for Young Scientists.



\end{document}